\documentclass[a4paper,11pt]{amsart}
\input{xypic}
\allowdisplaybreaks[4]
\newtheorem{proposition}{Proposition}[section]
\newtheorem{lemma}[proposition]{Lemma}
\newtheorem{corollary}[proposition]{Corollary}
\newtheorem{theorem}[proposition]{Theorem}

\theoremstyle{definition}

\theoremstyle{remark}
\newtheorem{remark}[proposition]{Remark}

\newcommand{\thlabel}[1]{\label{th:#1}}
\newcommand{\thref}[1]{Theorem~\ref{th:#1}}
\newcommand{\selabel}[1]{\label{se:#1}}
\newcommand{\seref}[1]{Section~\ref{se:#1}}
\newcommand{\lelabel}[1]{\label{le:#1}}
\newcommand{\leref}[1]{Lemma~\ref{le:#1}}
\newcommand{\prlabel}[1]{\label{pr:#1}}
\newcommand{\prref}[1]{Proposition~\ref{pr:#1}}
\newcommand{\colabel}[1]{\label{co:#1}}
\newcommand{\coref}[1]{Corollary~\ref{co:#1}}

\newcommand{\relabel}[1]{\label{re:#1}}
\newcommand{\reref}[1]{Remark~\ref{re:#1}}

\newcommand{\eqlabel}[1]{\label{eq:#1}}
\newcommand{\equref}[1]{(\ref{eq:#1})}

\def\equal#1{\smash{\mathop{=}\limits^{#1}}}

\newcommand{\can}{{\rm can}}

\newcommand{\Hom}{{\rm Hom}}
\newcommand{\HOM}{{\rm HOM}}
\newcommand{\End}{{\rm End}}

\newcommand{\END}{{\rm END}}

\def\ot{\otimes}

\newcommand{\Cc}{\mathcal{C}}
\newcommand{\Dd}{\mathcal{D}}
\newcommand{\Ee}{\mathcal{E}}

\newcommand{\Mm}{\mathcal{M}}

\newcommand{\Xx}{\mathcal{X}}
\def\*C{{}^*\hspace*{-1pt}{\Cc}}
\def\text#1{{\rm {\rm #1}}}

\def\ol{\overline}

\begin{document}
\title[Hopf-Galois extensions]{Hopf-Galois extensions and
isomorphisms of small categories}

\author[S. Caenepeel]{Stefaan Caenepeel}
\address{Faculty of Engineering,
Vrije Universiteit Brussel, Pleinlaan 2, B-1050 Brussels, Belgium}
\email{scaenepe@vub.ac.be}
\urladdr{http://homepages.vub.ac.be/\~{}scaenepe/}


\begin{abstract}
We associate two linear categories with two objects to a module over
the subalgebra of coinvariants of a Hopf-Galois extension, and prove that
they are isomorphic. The structure Theorem for cleft extensions, and
the Militaru-\c Stefan lifting Theorem can be obtained using these isomorphisms.
\end{abstract}

\maketitle

\section*{Introduction}
Our starting points are the following two classical results on Hopf algebras.
The first one is the structure theorem of cleft $H$-comodule algebras \cite[DT], stating that a 
cleft $H$-comodule algebra is isomorphic to a crossed product, and, conversely,
every crossed product is cleft. A comprehensive treatment can be found in \cite[Ch. 7]{Mont}.\\
The second result is the Militaru-\c Stefan lifting Theorem. Let $A$ be a 
faithfully flat Hopf-Galois
extension over its ring of coinvariants $B$, and $M$ a $B$-module. Generalizing
results due to Dade \cite{Dade} on strongly graded rings, Militaru and \c Stefan show that
the $B$-action on $M$ can be extended to an $A$-action if and only if there exists
an $H$-colinear algebra map between $H$ and the $A$-endomorphism ring of $M\ot_B A$.\\
Let us now explain the philosophy behind this note. A $k$-algebra can be viewed as a
$k$-linear category with one object. Isomorphisms between $k$-algebras can be obtained
from equivalences between $k$-linear categories. Examples of such equivalences come
from faithfully flat Hopf algebra extensions: then we have a pair of inverse equivalences
between modules over the ring of coinvariants and relative Hopf modules.\\
Now we consider ``double" $k$-algebras, namely $k$-linear categories with two objects.
For a right $H$-comodule algebra $A$, we introduce such a double algebra $\Cc_A$.
One of its endomorphism algebras consists of $k$-linear maps from $H$ to the coinvariants,
and on its homomorphism modules consists of $H$-colinear maps $H\to A$. This construction
is given in \seref{2}.\\
Given a module $M$ over the coinvariants $B$, we introduce another double algebra $\Dd_M$,
as the full subcategory of the category of $B$-modules and $H$-comodules, with objects
$M\ot H$ and $M\ot_B A$. Our main result, \thref{3.1} states that the categories
$\Cc_A$ and $\Dd_M$ are isomorphic if $A$ is a faithfully flat $H$-Galois extension of $B$.
In \seref{5}, we discuss how this category equivalence (or at least some variation of it)
can be applied the structure Theorem for cleft algebras, and in \seref{6}, we see how
the Militaru-\c Stefan lifting result can be obtained.

\section{Hopf-Galois extensions}\selabel{1}
Hopf-Galois theory was introduced in \cite{CS}, and later generalized in 
\cite{KT,Schneider0,Schneider1}. We recall the definitions and the most important results.
Let $H$ be a Hopf algebra over a commutative ring $k$, and assume that the antipode
$S$ is bijective. We use the Sweedler notation for the comultiplication: $\Delta(h)=
h_{(1)}\ot h_{(2)}$, for $h\in H$. If $M$ is a right $H$-comodule, then we use the following
notation for the coaction $\rho$: $\rho(m)=m_{[0]}\ot m_{[1]}$, for $m\in M$. In a similar
way, we write $\lambda(n)=n_{[-1]}\ot n_{[0]}$ for the left $H$-coaction on an element
$n$ in a left $H$-comodule $N$.\\
Let $A$ be a right $H$-comodule algebra, this is an algebra in the monoidal category of
right $H$-comodules. A relative right $(A,H)$-comodule is a right $A$-module that has
also the structure of a right $H$-comodule such that the compatibility relation
$$\rho(ma)=m_{[0]}a_{[0]}\ot m_{[1]}a_{[1]}$$
holds for all $m\in M$ and $a\in A$. $M^{{\rm co}H}=\{m\in M~|~\rho(m)=m\ot 1\}$
is the submodule of coinvariants, and is a right $B$-module, where $B=A^{{\rm co}A}$
is the subring of coinvariants of $A$. $\Mm_A^H$ is the category of relative Hopf modules,
and right $A$-linear $H$-colinear maps. We have a pair of adjoint functors $(F,G)$
between the categories $\Mm_B$ and $\Mm_A^H$. $F=-\ot_BA$ is the induction functor,
and $G=(-)^{{\rm co}A}$ is the coinvariants functor. The unit $\eta$ and counit $\varepsilon$ of the adjunction
are the following ($M\in \Mm_B$ and $N\in \Mm_A^H$):
\begin{eqnarray*}
&&\eta_M:\ M\to (M\ot_B A)^{{\rm co}A},~~~\eta_M(m)=m\ot_B 1;\\
&&\varepsilon_N:\ M^{{\rm co}A}\ot A\to M,~~~\varepsilon(m\ot_B a)=ma.
\end{eqnarray*}
The canonical map $\can$ associated to $A$ is defined by
$$\can:\ A\ot_B A\to A\ot H,~~\can(a\ot_B a')=aa'_{[0]}\ot a'_{[1]}.$$
If $\can$ is an isomorphism, then $A$ is called a {\sl Hopf-Galois extension}
or $H$-{\sl Galois extension} of $B$.\\
We can also consider left-right $(A,H)$-modules: these are $k$-modules with a left $A$-action and
a right $H$-coaction such that $\rho(am)=a_{[0]}m_{[0]}\ot a_{[1]}m_{[1]}$, for all $a\in A$ and
$m\in M$. We have a pair of adjoint functors
$(F'=A\ot_B-,G'=(-)^{{\rm co}H}$ between ${}_B\Mm$ and ${}_A\Mm^H$, the category of left-right
$(A,H)$-modules. The unit and counit are this time given by
\begin{eqnarray*}
\eta'_M:\ M\to (A\ot_B M)^{{\rm co}H},&&\eta'_M(m)=1\ot_B m;\\
\varepsilon'_N:\ A\ot_B N^{{\rm co}H}\to N,&&\varepsilon'_N(a\ot_B n)=an.
\end{eqnarray*}
The canonical map $\can':\ A\ot_B A\to A\ot H$ is defined by the formula
$$\can'(a\ot_B a')=a_{[0]}a'\ot a_{[1]}.$$
It is well-known that $\can$ is an isomorphism if and only if $\can'$ is an isomorphism: this
follows from the fact that $\can'=\Phi\circ \can$, with $\Phi:\ A\ot H\to A\ot H$ given by
$\Phi(a\ot h)=a_{[0]}\ot a_{[1]}S(h)$ and $\Phi^{-1}(a\ot h)=a_{[0]}\ot a_{[1]}\ol{S}(h)$

\begin{theorem}\thlabel{1.1}
Let $A$ be a right $H$-comodule algebra,
and consider the following statements:
\begin{enumerate}
\item $(F,G)$ is a pair of inverse equivalences;
\item $(F,G)$ is a pair of inverse equivalences and $A\in {}_{B}\Mm$ is flat;
\item $\can$ is an isomorphism and $A\in {}_B\Mm$ is faithfully flat;
\item $(F',G')$ is a pair of inverse equivalences;
\item $(F',G')$ is a pair of inverse equivalences and $A\in\Mm_B$ is flat;
\item $\can'$ is an isomorphism and $A\in\Mm_B$ is faithfully flat;
\end{enumerate}
Then $(3)\Longleftrightarrow (2)\Longrightarrow (1)$ and $(6)\Longleftrightarrow (5)\Longrightarrow (4)$.
If $H$ is flat as a $k$-module, then $(1)\Longleftrightarrow (2)$ and $(4)\Longleftrightarrow (5)$.
If $k$ is a field, then the six statements are equivalent.
\end{theorem}

Let $A$ be a faithfully flat right $H$-Galois extension. The inverse of the canonical map
$\can$ is completely determined by the map
$$\gamma_A=\can^{-1}\circ(\eta_A\ot H):\ H\to A\ot_{B}A,\quad h\mapsto
\sum_i l_i(h)\ot_{B}r_i(h).$$ 
Then the element $\gamma_A(h)$ is
characterized by the property
\begin{equation}\eqlabel{1.2.1}
\sum_i l_i(h)r_i(h)_{[0]}\ot r_i(h)_{[1]}=1\ot h.
\end{equation}
For all $h,h'\in H$ and $a\in A$, we have (see \cite[3.4]{Schneider1}):
\begin{eqnarray}
&&\gamma_A(h)\in (A\ot_{B}A)^{B};\eqlabel{1.2.2}\\
&&\gamma_A(h_{(1)})\ot h_{(2)}=
\sum_i l_i(h)\ot_{B} r_i(h)_{[0]}\ot r_i(h)_{[1]};\eqlabel{1.2.3}\\
&&\gamma_A(h_{(2)})\ot S(h_{(1)})=
\sum_i l_i(h)_{[0]}\ot_{B} r_i(h)\ot l_i(h)_{[1]};\eqlabel{1.2.4}\\
&&\sum_i l_i(h)r_i(h)=\varepsilon(h)1_A;\eqlabel{1.2.5}\\
&&\sum_i a_{[0]}l_i(a_{[1]})\ot_B r_i(a_{[1]})=1\ot_B a;\eqlabel{1.2.6}\\
&&\sum_i l_i(\ol{S}(a_{[1]}))\ot_B r_i(\ol{S}(a_{[1]}))a_{[0]}=a\ot_B1;
\eqlabel{1.2.6a}\\
&&\gamma_A(hh')=\sum_{i,j} l_i(h')l_j(h)\ot_{B}
r_j(h)r_i(h').\eqlabel{1.2.7}
\end{eqnarray}

\section{The categories $\Cc_A$ and $\Cc'_A$}\selabel{2}
Let $A$ be a right $H$-comodule algebra, and $B=A^{{\rm co}A}$, as in \seref{1}.
We introduce a category $\Cc_A$, with two objects ${\bf 1}$ and ${\bf 2}$. The
morphisms are defined as follows.
\begin{eqnarray*}
\Cc_A({\bf 1},{\bf 1})&=&\Hom(H,B)\\
&=&\{v:\ H\to A~|~\rho(v(h))=v(h)\ot 1,~~{\rm for~all~}h\in H\};\\
\Cc_A({\bf 2},{\bf 1})&=&\Hom^H(H,A)\\
&=&\{t:\ H\to A~|~\rho(t(h))=t(h_{(1)})\ot h_{(2)},~~{\rm for~all~}h\in H\};\\
\Cc_A({\bf 1},{\bf 2})&=&\{u:\ H\to A~|~\rho(u(h))=u(h_{(2)})\ot S(h_{(1)}),~~{\rm for~all~}h\in H\};\\
\Cc_A({\bf 2},{\bf 2})&=&\{w:\ H\to A~|~\rho(w(h))=w(h_{(2)})\ot S(h_{(1)})h_{(3)},~~{\rm for~all~}h\in H\}.
\end{eqnarray*}
The composition of morphisms is given by the convolution on $\Hom(H,A)$. We have to
verify that, for $f:\ {\bf i}\to {\bf j}$ and $g:\ {\bf j}\to {\bf k}$, then $g*i\in \Cc_A({\bf i}, {\bf k})$. Let us
do this in the case where ${\bf i}={\bf j}={\bf k}={\bf 2}$: for $w,w_1\in \Cc_A({\bf 2},{\bf 2})$ and $h\in H$,
we have
\begin{eqnarray*}
\rho(w*w_1)(h)&=& \rho(w(h_{(1)})w_1(h_{(2)})\\
&=&w(h_{(2)})w_1(h_{(5)})\ot S(h_{(1)})h_{(3)}S(h_{(4)})h_{(6)}\\
&=& w(h_{(2)})w_1(h_{(3)})\ot S(h_{(1)})h_{(4)}\\
&=& (w*w_1)(h_{(2)})\ot S(h_{(1)})h_{(3)},
\end{eqnarray*}
and it follows that $w*w_1\in \Cc_A({\bf 2},{\bf 2})$, as needed. Verification in all the other cases is similar
and is left to the reader.\\
We also introduce the category $\Cc'_A$ and show that it is isomorphic
to $\Cc_A$. It is introduced because it allows us to simplify slightly some of the computations
in \seref{3}. $\Cc'_A$ also has two objects, ${\bf 1}$ and ${\bf 2}$. The morphisms are defined
in the following fashion.
$$\begin{array}{l}
\Cc'_A({\bf 1},{\bf 1})=\Hom(H,B)\\
\hspace*{15mm}=\{v': H\to A~|~\rho(v'(h))=v'(h)\ot 1,~{\rm for~all~}h\in H\};\\
\Cc'_A({\bf 1},{\bf 2})=\Hom^H(H,A)\\
\hspace*{15mm}=\{t': H\to A~|~\rho(t'(h))=t'(h_{(1)})\ot h_{(2)},~{\rm for~all~}h\in H\};\\
\Cc'_A({\bf 2},{\bf 1})=\{u': H\to A~|~\rho(u'(h))=u'(h_{(2)})\ot \ol{S}(h_{(1)}),~{\rm for~all~}h\in H\};\\
\Cc'_A({\bf 2},{\bf 2})=\{w': H\to A\,|\,\rho(w'(h))=w'(h_{(2)})\ot h_{(3)}\ol{S}(h_{(1)}),~{\rm for~all~}h\in H\}.
\end{array}$$
The composition of two morphisms in $\Cc'_A$ is given by the 
convolution product in $\Hom(H^{\rm cop},A)$:
$$(f'\star g')(h)=f'(h_{(2)})g'(h_{(1)}).$$

\begin{proposition}\prlabel{2.1}
We have an isomorphism of categories $\gamma:\ \Cc'_A\to \Cc_A$, which is the identity
at the level of objects. At the level 
of morphisms, it is given by $\gamma(f')=f'\circ S$.
\end{proposition}

\begin{proof}
We have to show first that $\gamma(\Cc'_A({\bf i},{\bf j}))\subset
\Cc_A({\bf i},{\bf j})$. Let us do this in the case ${\bf i}={\bf j}={\bf 2}$, the other cases
are done in a similar way. So take $w'\in \Cc'_A({\bf 2},{\bf 2})$, and let $w=
w'\circ S=\gamma(w')$. Then for all $h\in H$, we have that
\begin{eqnarray*}
\rho(w(h))&=& \rho(w'(S(h)))=w'(S(h)_{(2)})\ot S(h)_{(3)}\ol{S}(S(h)_{(1)})\\
&=& w'(S(h_{(2)}))\ot S(h_{(1)})h_{(3)}= w(h_{(2)})\ot S(h_{(1)})h_{(3)},
\end{eqnarray*}
proving that $w\in \Cc_A({\bf 2},{\bf 2})$, as needed. It is easy to see that $\gamma$ 
respects the composition of morphisms:
\begin{eqnarray*}
&&\hspace*{-2cm}
\gamma(f'\star g')(h)=(f'\star g')(S(h))=
f'(S(h_{(1)}))g'(S(h_{(2)}))\\
&=&f(h_{(1)})g(h_{(2)})= (f*g)(h).
\end{eqnarray*}
Finally, $\gamma$ is an isomorphism. The inverse functor $\ol{\gamma}$ is given by
$\ol{\gamma}(f)=f\circ \ol{S}$.
\end{proof}

The functor $\gamma$ induces maps $\gamma_{ji}:\ \Cc'_A({\bf i},{\bf j})\to
\Cc_A({\bf i},{\bf j})$.

\section{Main result}\selabel{3}
Let $A$ be a faithfully flat right $H$-Galois extension. We assume moreover that $H$ is projective
as a $k$-module. This is always satisfied if we work over a field $k$. Let $P$ and $Q$ be
two right relative Hopf modules. We have a map
$$\rho:\ \Hom_A(P,Q)\to \Hom_A(P,Q\ot H),
~~\rho(f)(p)=f(p_{[0]})_{[0]}\ot f(p_{[0]})_{[1]}S(p_{[1]})$$
As $H$ is projective, the natural map $\Hom_A(P,Q)\ot H\to \Hom_A(P,Q\ot H)$ is a monomorphism,
and we can consider $\Hom_A(P,Q)\ot H$ as a submodule of $H\to \Hom_A(P,Q\ot H)$
We call $f\in \Hom_A(P,Q)$ {\sl rational} if $\rho(f)\in \Hom_A(P,Q)\ot H$, that is, if there exists
an element $f_{[0]}\ot f_{[1]}\in \Hom_A(P,Q)\ot H$ (summation implicitely understood) such
that
$\rho(f)(p)=f_{[0]}(p)\ot f_{[1]}$, for all $p\in P$, which is equivalent to
\begin{equation}\eqlabel{14}
\rho(f(p))=f_{[0]}(p_{[0]})\ot f_{[1]}p_{[1]}.
\end{equation}
The submodule of $\Hom_A(P,Q)$ consisting of all rational maps is denoted by $\HOM_A(P,Q)$,
and is a right $H$-comodule.
$\END_A(P)$ is a right $H$-comodule algebra. Now we take $P=M\ot_B A$, where $M\in \Mm_B$,
$E=\END_A(M\ot_BA)$ and 
$$F=E^{{\rm co}H}=\END_A^H(M\ot_B A)\cong \End_B(M),$$
in view of \thref{1.1}. Then we can consider the categories $\Cc_E$
and $\Cc'_E$, as in \seref{2}.\\
We have seen in \seref{1} that $M\ot_BA\in \Mm_A^H$ is a relative Hopf module. In particular,
it is also an object in $\Mm_B^H$, where $B$ is considered as a right $H$-comodule algebra
with trivial $H$-coaction. In fact $\Mm_B^H$ is the category of right $B$-modules with a right
$H$-coaction such that $\rho(mb)=m_{[0]}\ot m_{[1]}b$, for all $m\in M$ and $b\in B$.
$M\ot H$ is also an object of  $\Mm_B^H$, with $B$-action and $H$-coaction given by
$\rho(m\ot h)=m\ot \Delta(h)$ and $(m\ot h)b=mb\ot h$.\\
Now let $\Dd_M$ be the full subcategory of $\Mm_B^H$ with objects $M\ot_B A$ and
$M\ot H$. Out main result is the following.

\begin{theorem}\thlabel{3.1}
Let $H$ be a projective Hopf algebra, and
$A$ a faithfully flat right $H$-Galois extension. For $M\in \Mm_B$, we have a commutative
diagram of isomorphisms of categories:
$$\xymatrix{
\Cc'_E\ar[dr]_{\beta}\ar[rr]^{\gamma}&&\Cc_E\ar[dl]^{\alpha}\\
&\Dd_M&}$$
\end{theorem}

At the level of morphisms, the functors $\alpha$ and $\alpha'$ are defined in the obvious way:
$$\alpha({\bf 1})=\alpha'({\bf 1})=M\ot H~~:~~\alpha({\bf 2})=\alpha'({\bf 2})=M\ot_B A.$$
In the subsequent Lemmas, we will define $\alpha$ and $\alpha'$ at the level of morphisms. The proof
of the following result is straightforward, and is left to the reader.

\begin{lemma}\lelabel{3.2}
We have an isomorphism of $k$-modules
$$\delta_1:\ \Hom_B(M\ot_B A,M)\to \Hom_B^H(M\ot_BA,M\ot H),$$
given by
$$\delta_1(\phi)(m\ot_Ba)=\phi(m\ot_Ba_{[0]})\ot a_{[1]}~~;~~
\ol{\delta}_1(\varphi)=(M\ot\varepsilon)\circ\varphi.$$
We have an isomorphism of $k$-algebras
$$\delta_2:\ \Hom_B(M\ot H,M)\to \End_B^H(M\ot H),$$
given by
$$\delta_2(\Theta)(m\ot h)=\Theta(m\ot h_{(1)})\ot h_{(2)}~~;~~
\ol{\delta}_2(\theta)=(M\ot\varepsilon)\circ \theta.$$
The multiplication on $\Hom_B(M\ot H,M)$ is given by the formula
$\Theta\cdot \Theta'=\Theta\circ \delta_2(\Theta')$, or, more explicitly,
\begin{equation}\eqlabel{3.2.1}
(\Theta\cdot \Theta')(m\ot h)=\Theta(\Theta'(m\ot h_{(1)})\ot h_{(2)}).
\end{equation}
\end{lemma}

\begin{lemma}\lelabel{3.3}
We have an algebra map
$$\tilde{\beta}_{11}:\ \Cc'_E({\bf 1},{\bf 1})=\Hom(H,F)\to \Hom_B(M\ot H, M),$$
given by
$$\tilde{\beta}_{11}(v')(m\ot h)=\eta_M^{-1}\bigl(v'(h)(m\ot_B 1)\bigr).$$
\end{lemma}

\begin{proof}
For all $h\in H$, we have that $v'(h)\in F=E^{{\rm co}H}$. Using \equref{14}, we find
that
$$\rho\bigl(v'(h)(m\ot_B 1)\bigr)= v'(h)(m\ot_B 1)\ot 1,$$
hence $v'(h)(m\ot_B 1)\in (M\ot_B A)^{{\rm co}H}$. We know from \thref{1.1} that
$\eta_M:\ M\to (M\ot_B A)^{{\rm co}H}$ is an isomorphism, so that
$\tilde{\beta}_{11}$ is well-defined, and is characterized by the formula
\begin{equation}\eqlabel{3.3.1}
\tilde{\beta}_{11}(v')(m\ot h)\ot_B 1= v'(h)(m\ot_B1).
\end{equation}
Let us now show that $\tilde{\beta}_{11}(v')$ is right $B$-linear. For all $m\in M$,
$b\in B$ and $h\in H$, we have
\begin{eqnarray*}
&&\hspace*{-2cm}
\tilde{\beta}_{11}(v')(mb\ot h)\ot_B 1= v'(h)(mb\ot_B1)= v'(h)(m\ot_B1)b\\
&=& \tilde{\beta}_{11}(v')(m\ot h)\ot_B b=\tilde{\beta}_{11}(v')(m\ot h)b\ot_B 1.
\end{eqnarray*}
We will now show that $\tilde{\beta}_{11}$ has an inverse, given by
$$\bigl(\hat{\beta}_{11}(\Theta)(h)\bigr)(m\ot_B a)=
\Theta(m\ot h)\ot_B a.$$
We have to show first that $\hat{\beta}_{11}$ is well-defined, that is,
$\hat{\beta}_{11}(h)\in F$, for all $h\in H$. To this end, we compute that
\begin{eqnarray*}
&&\hspace*{-2cm}
\rho\Bigl(\bigl(\hat{\beta}_{11}(\Theta)(h)\bigr)(m\ot_B a)\Bigr)
= \Theta(m\ot h)\ot_B a_{[0]}\ot a_{[1]}\\
&=& \bigl(\hat{\beta}_{11}(\Theta)(h)\bigr)(m\ot_B a_{[0]})\ot a_{[1]},
\end{eqnarray*}
and conclude from \equref{14} that $\rho\bigl(\hat{\beta}_{11}(\Theta)(h)\bigr)=
\hat{\beta}_{11}(\Theta)(h)\ot 1$.\\
We now show that $\tilde{\beta}_{11}$ and $\hat{\beta}_{11}$ are inverses.
For all $\Theta\in \Hom_B(M\ot H, M)$, $v'\in \Hom(H,F)$
$m\in M$, $h\in H$ and $a\in A$, we have
\begin{eqnarray*}
&&\hspace*{-2cm}
\tilde{\beta}_{11}\bigl(\hat{\beta}_{11}(\Theta)\bigr)(m\ot h)\ot_B 1
= (\hat{\beta}_{11}(\Theta)(h)\bigr)(m\ot_B 1)\\
&=& \Theta(m\ot h)\ot_B 1;\\
&&\hspace*{-2cm}\bigl(\hat{\beta}_{11}(\tilde{\beta}_{11}(v'))(h)\bigr)(m\ot_B a)
= (\tilde{\beta}_{11}(v'))(m\ot h)\ot_B a\\
&=& v'(h)(m\ot_B 1)a=v'(h)(m\ot_B 1).
\end{eqnarray*}
Let us finally show that $\tilde{\beta}_{11}$ is an algebra map. For $v',v'_1:\ H\to F$,
$m\in M$ and $h\in H$, we have
\begin{eqnarray*}
&&\hspace*{-1cm}
\bigl(\tilde{\beta}_{11}(v')\cdot \tilde{\beta}_{11}(v'_1)\bigr)(m\ot h)\ot_B 1
=
\tilde{\beta}_{11}(v')\bigl(\tilde{\beta}_{11}(v'_1)(m\ot h_{(1)})\ot h_{(2)}\bigr)\ot 1\\
&=& v'(h_{(2)})\bigl(\tilde{\beta}_{11}(v'_1)(m\ot h_{(1)})\ot_B 1\bigr)
= (v'(h_{(2)})\circ v'_1(h_{(1)}))(m\ot_B 1)\\
&=&(v'\star v'_1)(h)(m\ot_B 1)
=\tilde{\beta}_{11}(v'\star v'_1)(m\ot h)\ot_B 1,
\end{eqnarray*}
and it follows that $\tilde{\beta}_{11}(v'\star v'_1)=\tilde{\beta}_{11}(v')\cdot \tilde{\beta}_{11}(v'_1)$.
\end{proof}

\begin{corollary}\colabel{3.4}
We have algebra isomorphisms
$$\beta_{11}=\delta_2\circ \tilde{\beta}_{11}:\
 \Cc'_E({\bf 1},{\bf 1})\to \End_B^H(M\ot H);$$
 $$\alpha_{11}=\delta_2\circ \tilde{\beta}_{11}\circ \gamma_{11}^{-1}:\
 \Cc_E({\bf 1},{\bf 1})\to \End_B^H(M\ot H).$$
 \end{corollary}

\begin{lemma}\lelabel{3.5}
We have an isomorphism of $k$-modules
$$\beta_{21}:\ \Cc'_E({\bf 1},{\bf 2})=\Hom(H,E)\to \Hom_B^H(M\ot H,M\ot_B A),$$
given by
$$\beta_{21}(t')(m\ot h)=t'(h)(m\ot_B 1),$$
for $t'\in \Hom(H,E)$, $m\in M$, $h\in H$.
Consequently, we also have an isomorphism
$$\alpha_{21}=\beta_{21}\circ\gamma_{21}^{-1}:\ \Cc_E({\bf 1},{\bf 2})\to \Hom_B^H(M\ot H,M\ot_B A).$$
\end{lemma}

\begin{proof}
It is easy to see that $\beta_{21}(t')$ is right $A$-linear:
\begin{eqnarray*}
&&\hspace*{-2cm}
\beta_{21}(t')(mb\ot h)=t'(h)(mb\ot_B 1)=t'(h)(m\ot_B b)\\
&=&t'(h)(m\ot_B 1)b=(\beta_{21}(t')(m\ot h))b.
\end{eqnarray*}
$\beta_{21}(t')$ is right $H$-colinear:
\begin{eqnarray*}
&&\hspace*{-2cm}
\rho\bigl(\beta_{21}(t')(m\ot h)\bigr)
= \rho\bigl(t'(h)(m\ot_B 1)\bigr)
= t'(h)_{[0]}(m\ot_B 1)\ot t'(h)_{[1]}\\
&=& t'(h_{(1)})(m\ot_B 1)\ot h_{(2)}=\beta_{21}(t')(m\ot h_{(1)})\ot h_{(2)}.
\end{eqnarray*}
This shows that $\beta_{21}(t')\in \Hom_B^H(M\ot H,M\ot_B A)$, as needed.
Now we define a map
$$\ol{\beta}_{21}:\  \Hom_B^H(M\ot H,M\ot_B A)\to \Hom(H,E)$$
by the formula
$$(\ol{\beta}_{21}(\psi))(h)(m\ot_B a)=\psi(m\ot h)a.$$
We first show that $\ol{\beta}_{21}$ is well-defined, and then that it is inverse to $\beta_{21}$.\\
$\ol{\beta}_{21}(\psi)$ is right $H$-colinear: we first compute
\begin{eqnarray*}
&&\hspace*{-2cm}
\rho\bigl((\ol{\beta}_{21}(\psi))(h)(m\ot_B a)\bigr)
=\rho\bigl(\psi(m\ot h)a\bigr)\\
&=& \psi(m\ot h_{(1)})a_{[0]}\ot h_{(2)}a_{[1]}\\
&=& (\ol{\beta}_{21}(\psi))(h_{(1)})(m\ot_B a_{[0]})\ot h_{(2)}a_{[1]},
\end{eqnarray*}
and we conclude from \equref{14} that $\rho\bigl((\ol{\beta}_{21}(\psi))(h)\bigr)=(\ol{\beta}_{21}(\psi))
(h_{(1)})\ot h_{(2)}$, as needed. Let us finally show that $\beta_{21}$ and $\ol{\beta}_{21}$
are inverses. For all $t'\in \Hom^H(H,E)$, $\psi\in \Hom_B^H(M\ot H,M\ot_B A)$,
$m\in M$, $a\in A$ and $h\in H$, we have
\begin{eqnarray*}
&&\hspace*{-2cm}
(\beta_{21}\circ \ol{\beta}_{21})(\psi)(m\ot h)=
( \ol{\beta}_{21}(h))(m\ot_B 1)\\
&=& \psi(m\ot_B 1)a=\psi(m\ot_B a);\\
&&\hspace*{-2cm}
\Bigl(\bigl((\ol{\beta}_{21}\circ {\beta}_{21})(t')\bigr)(h)\Bigr)(m\ot_B a)
=(\beta_{21}(t')(m\ot h)a\\
&=& t'(h)(m\ot_B 1)a=t'(h)(m\ot_B a).
\end{eqnarray*}
\end{proof}

\begin{lemma}\lelabel{3.6}
We have an isomorphism of $k$-modules
$$\tilde{\beta}_{12}:\ \Cc'_E({\bf 2},{\bf 1})\to \Hom_B(M\ot_BA, M),$$
given by
$$\tilde{\beta}_{12}(u')(m\ot_B a)=\eta_M^{-1}\bigl(u'(a_{[1]})(m\ot_B a_{[0]})\bigr).$$
\end{lemma}

\begin{proof}
First, we have to show that $u'(a_{[1]})(m\ot_B a_{[0]})\in (M\ot_B A)^{{\rm co}H}$.
This can be seen as follows:
\begin{eqnarray*}
&&\hspace*{-2cm}
\rho\bigl(u'(a_{[1]})(m\ot_B a_{[0]})\bigr)
=u'(a_{(3)})(m\ot_B a_{[0]})\ot \ol{S}(a_{(2)})a_{[1]}\\
&=& u'(a_{[1]})(m\ot_B a_{[0]})\ot 1.
\end{eqnarray*}
Remark that $\tilde{\beta}_{12}(u')(m\ot_B a)$ is characterized by the formula
\begin{equation}\eqlabel{3.6.0}
\tilde{\beta}_{12}(u')(m\ot_B a)\ot_B 1= u'(a_{[1]})(m\ot_B a_{[0]}).
\end{equation}
Now we show that $\tilde{\beta}_{12}(u')$ is right $B$-linear: for $b\in B$,
we have
\begin{eqnarray*}
&&\hspace*{-2cm}
\tilde{\beta}_{12}(u')(m\ot_B ab)\ot_B 1= u'(a_{[1]})(m\ot_B a_{[0]}b)=
 u'(a_{[1]})(m\ot_B a_{[0]})b\\
 &=&\tilde{\beta}_{12}(u')(m\ot_B a)\ot_B b
=\tilde{\beta}_{12}(u')(m\ot_B a)b\ot_B 1.
\end{eqnarray*}
Now we construct a map
$$\hat{\alpha}_{12}: \Hom_B(M\ot_BA, M)\to \Cc_E({\bf 2},{\bf 1})=\Hom^H(H,E).$$
as follows: 
\begin{equation}\eqlabel{AA}
\bigl(\hat{\alpha}_{12}(\phi)(h)\bigr)(m\ot_Ba)=\sum_i
\phi(m\ot l_i(h))\ot_B r_i(h)a.
\end{equation}
It is clear that $\bigl(\hat{\alpha}_{12}(\phi)\bigr)(h)$ is right $A$-linear. Then we need to show
that $\hat{\alpha}_{12}(\phi)$ is right $H$-colinear. To this end, we need to show that
\begin{equation}\eqlabel{3.6.1}
\rho\bigl(\hat{\alpha}_{12}(\phi)(h)\bigr)= \hat{\alpha}_{12}(\phi)(h_{(1)})\ot h_{(2)},
\end{equation}
for all $h\in H$. For all $m\in M$ and $a\in A$, we compute
\begin{eqnarray*}
&&\hspace*{-2cm}
\rho\Bigr(\bigl(\hat{\alpha}_{12}(\phi)(h)\bigr)(m\ot_B a)\Bigr)\\
&=& \sum_i \phi(m\ot l_i(h))\ot_B r_i(h)_{[0]}a_{[0]}\ot r_i(h)_{[1]}a_{[1]}\\
&\equal{\equref{1.2.3}}&
\sum_i \phi(m\ot l_i(h_{(1)}))\ot_B r_i(h_{(1)})a_{[0]}\ot h_{(2)}a_{[1]}\\
&=& \bigl(\hat{\alpha}_{12}(\phi)(h_{(1)})\bigr)(m\ot_B a_{[0]})\ot h_{(2)}a_{[1]},
\end{eqnarray*}
and \equref{3.6.1} follows as an application of \equref{14}.\\
Now we define $\hat{\beta}_{12}=\hat{\alpha}_{12}\circ \gamma_{12}^{-1}$, and show
that $\hat{\beta}_{12}$ and $\hat{\alpha}_{12}$ are inverses. $\hat{\beta}_{12}$ is
given by the formula
$$(\hat{\beta}_{12}(\phi)(h))(m\ot_B a)=\sum_i\phi(m\ot_B l_i(\ol{S}(h))\ot_B r_i(\ol{S}(h))a.$$
Now we compute
\begin{eqnarray*}
&&\hspace*{-2cm}
\Bigl(\bigl((\hat{\beta}_{12}\circ\tilde{\beta}_{12})(u')\bigr)(h)\Bigr)(m\ot_B a)\\
&=&
\sum_i (\tilde{\beta}_{12}(u'))\bigl(m\ot_B l_i(\ol{S}(h)\bigr)\ot_B r_i(\ol{S}(h))a\\
&=& (u'(l_i(\ol{S}(h)_{[1]}))\bigl(m\ot_B l_i(\ol{S}(h)_{[0]}\bigr)r_i(\ol{S}(h)a\\
&\equal{\equref{1.2.4}}&
(u'(S(\ol{S}(h_{(2)}))))\bigl(m\ot_B l_i(\ol{S}(h_{(1)})\bigr)r_i(\ol{S}(h_{(1)})a\\
&=&
(u'(h_{(2)}))\bigl(m\ot_B l_i(\ol{S}(h_{(1)})r_i(\ol{S}(h_{(1)})a\bigr)\\
&\equal{\equref{1.2.5}}& u'(h)(m\ot_B a);\\
&&\hspace*{-2cm}
\bigl((\tilde{\beta}_{12}\circ\hat{\beta}_{12})(\phi)\bigr)(m\ot_B a)\ot_B 1\\
&=&(\hat{\beta}_{12}(\phi))(a_{[1]})(m\ot_B a_{[0]})\\
&=& \sum_i\phi(m\ot_B l_i(\ol{S}(a_{[1]})))\ot_B r_i(\ol{S}(a_{[1]}))a_{[0]}\\
&\equal{\equref{1.2.6a}}& \phi(m\ot_B a)\ot_B 1.
\end{eqnarray*}
\end{proof}

\begin{corollary}\colabel{3.7}
We have $k$-module isomorphisms
$$\beta_{12}=\delta_1\circ \tilde{\beta}_{12}:\ \Cc'_E({\bf 2},{\bf 1})\to \Hom_B^H(M\ot_B A,M\ot H));$$
$$\alpha_{12}=\delta_1\circ \tilde{\beta}_{12}\circ \gamma_{12}^{-1}:\ \Cc_E({\bf 2},{\bf 1})\to \Hom_B^H(M\ot_B A,M\ot H)).$$
\end{corollary}

\begin{lemma}\lelabel{3.8}
We have an algebra isomorphism $\beta_{22}:\ \Cc'_E({\bf 2},{\bf 2})\to \End_B^H(M\ot_B A)$, given by
the formula
$$(\beta_{22}(w'))(p)=w'(p_{[1]})(p_{[0]}),$$
for all $p\in M\ot_B A$. Consequently, we also have an algebra isomorphism
$\alpha_{22}=\beta_{22}\circ \gamma_{22}^{-1}:\ \Cc'_E({\bf 2},{\bf 1})\to \End_B^H(M\ot_B A)$.
\end{lemma}

\begin{proof}
We first show that $\beta_{22}(w')$ is right $B$-linear. For $p\in M\ot_B A$ and $b\in B$, we have
$\rho(pb)=p_{[0]}b\ot p_{[1]}$, and
$$(\beta_{22}(w'))(pb)=w'(p_{[1]})(p_{[0]}b)=w'(p_{[1]})(p_{[0]})b.$$
$\beta_{22}(w')$ is right $H$-co-linear. Since $w'\in \Cc'_E({\bf 2},{\bf 1})$, we have
$$\rho(w'(h))=w(h_{(2)})\ot h_{(3)}\ol{S}(h_{(1)}),$$
hence
\begin{equation}\eqlabel{3.8.1}
\rho(w'(h)(p))=w(h_{(2)})(p_{[0]})\ot h_{(3)}\ol{S}(h_{(1)})p_{[1]}.
\end{equation}
Now we have
\begin{eqnarray*}
&&\hspace*{-2cm}
\rho\bigl((\beta_{22}(w'))(p)\bigr)=\rho\bigl(w'(p_{[1]})(p_{[0]})\bigr)
\equal{\equref{3.8.1}}
w'(p_{(3)})(p_{[0]})\ot p_{[4]}\ol{S}(p_{[2]})p_{[1]}\\
&=&w'(p_{(1)})(p_{[0]})\ot p_{[2]}= (\beta_{22}(w'))(p_{[0]})\ot p_{[1]}.
\end{eqnarray*}
We next show that $\beta_{22}$ is an algebra morphism, that is, it preserves multiplication and unit.
Mulitplication:
\begin{eqnarray*}
&&\hspace*{-2cm}
(\beta_{22}(w'\star w'_1))(p)= ((w'\star w'_1)(p_{[1]})(p_{[0]})\\
&=& (w'(p_{[2]})\circ w'_{[1]}(p_{[1]}))(p_{[0]})\\
&=& w'(p_{[1]})\bigl(\beta_{22}(w'_1)(p_{[0]})\bigr)\\
&=& w'\Bigl(\bigl(\beta_{22}(w'_1)(p)\bigr)_{[1]}\Bigr) \Bigl(\bigl(\beta_{22}(w'_1)(p)\bigr)_{[0]}\Bigr) \\
&=& \beta_{22}(w')\bigl( \beta_{22}(w'_1)(p)\bigr)\\
&=& (\beta_{22}(w')\circ \beta_{22}(w'_{1}))(p).
\end{eqnarray*}
In the fourth equality we used the fact that $\beta_{22}(w'_1)$ is right $H$-colinear.\\
Unit: $(\beta_{22}(\eta_E\circ \varepsilon_H))(p)=(\eta(\varepsilon(p_{[1]}))(p_{[0]})=p$.\\

Now we consider the map
$$\ol{\alpha}_{22}:\ \End_B^H(M\ot_B A)\to \Cc_E({\bf 2},{\bf 2}),$$
defined as follows: for $\kappa\in \End_B^H(M\ot_B A)$, let
$$\bigl(\ol{\alpha}_{22}(\kappa)(h)\bigr)(m\ot_B a)
=\sum_i \kappa(m\ot_B l_i(h))r_i(h)a.$$
We have to show that $\ol{\alpha}_{12}(\kappa)\in \Cc_E({\bf 2},{\bf 2})$, that is,
\begin{equation}\eqlabel{3.8.2}
\rho\bigl(\ol{\alpha}_{12}(\kappa)(h)\bigr)=
\bigl(\ol{\alpha}_{12}(\kappa)(h_{(2)})\bigr)\ot S(h_{(1)})h_{(3)}.
\end{equation}
We proceed as follows: for all $m\in M$ and $a\in A$, we have
\begin{eqnarray*}
&&\hspace*{-2cm}
\rho\bigl(\ol{\alpha}_{12}(\kappa)(h)(m\ot_B a\bigr)
= \rho\bigl( \sum_i \kappa(m\ot_B l_i(h))r_i(h)a\bigr)\\
&=& \sum_i \kappa\bigl(m\ot_B l_i(h)_{[0]}\bigr)r_i(h)_{[0]}a_{[0]} \ot 
l_i(h)_{[1]}r_i(h)_{[1]}a_{[1]} \\
&\equal{\equref{1.2.3}}&
\sum_i \kappa\bigl(m\ot_B l_i(h_{(1)})_{[0]}\bigr)r_i(h_{(1)})a_{[0]} \ot 
l_i(h_{(1)})_{[1]}h_{(2)}a_{[1]} \\
&\equal{\equref{1.2.4}}&
\sum_i \kappa\bigl(m\ot_B l_i(h_{(2)})\bigr)r_i(h_{(2)})a_{[0]} \ot 
S(h_{(1)})h_{(3)}a_{[1]} \\
&=& \ol{\alpha}_{22}(\kappa)(h_{(2)})(m\ot_B a_{[0]})\ot S(h_{(1)})h_{(3)}a_{[1]} 
\end{eqnarray*}
In the second equality, we used that $\kappa$ is right $H$-colinear.
\equref{3.8.2} then follows as an application of \equref{14}. Let us now show
that $\ol{\beta}_{22}=\ol{\alpha}_{12}\circ \gamma_{22}^{-1}$ and $\beta_{22}$
are inverses.
\begin{eqnarray*}
&&\hspace*{-2cm}
\bigl((\beta_{22}\circ \ol{\beta}_{22})(\kappa)\bigr)(m\ot_B a)=
(\ol{\beta}_{22}(\kappa)(a_{[1]}))(m\ot_B a_{[0]})\\
&=& \kappa\bigl( m\ot_B l_i(\ol{S}(a_{[1]}))r_i(\ol{S}(a_{[1]})a_{[0]}\equal{\equref{1.2.6a}}
\kappa(m\ot_B a);\\
&&\hspace*{-2cm}
\bigl(((\ol{\beta}_{22}\circ {\beta}_{22})(w'))(h)\bigr)(m\ot_B a)\\
&=&\sum_i \beta_{22}(w')(m\ot_B l_i(\ol{S}(h)))r_i(\ol{S}(h))a\\
&=& \sum_i \bigl(w'\bigl(l_i(\ol{S}(h))_{[1]}\bigr)\bigr)\bigl(m\o_B l_i(\ol{S}(h))_{[0]}\bigr)
r_i(\ol{S}(h))a\\
&\equal{\equref{1.2.4}}&
\sum_i w'\bigl( S(\ol{S}(h_{(2)})\bigr) \bigl(m\ot_B l_i(\ol{S}(h_{(1)})\bigr)r_i(\ol{S}(h_{(1)})a\\
&=& \sum_i w'(h_{(2)})\bigl(m\ot_B l_i(\ol{S}(h_{(1)})r_i(\ol{S}(h_{(1)})a\bigr)\\
&\equal{\equref{1.2.5}}& w'(h)(m\ot_B a).
\end{eqnarray*}
\end{proof}

\begin{proof} ({\sl of \thref{3.1}})
In the preceding Lemmas, we have shown that there exist isomorphisms
$$\xymatrix{
\Cc'_E({\bf i},{\bf j})\ar[r]^{\gamma_{ji}}&
\Cc_E({\bf i},{\bf j})\ar[r]^(.35){\alpha_{ji}}&
\Hom_B^H(\alpha({\bf i}),\alpha({\bf j})}$$
The proof of \thref{3.1} will be finished if we can show that, given
$f:\ {\bf i}\to {\bf j}$ and $g:\ {\bf j}\to {\bf k}$ in $\Cc_E$, we have
\begin{equation}\eqlabel{3.9.1}
\alpha_{kj}(g)\circ\alpha_{ji}(f)=\alpha_{ki}(g*f)
\end{equation}
We already know that \equref{3.9.1} holds if ${\bf i}= {\bf j}= {\bf k}$, see \coref{3.4}
and \leref{3.8}.\\
We now fix the following notation.
$$\begin{array}{ccc}
v'\in \Cc'_E({\bf 1},{\bf 1})& v=\gamma_{11}(v')\in \Cc_E({\bf 1},{\bf 1})&
\theta=\alpha_{11}(v):\ M\ot H\to M\ot H\\
t'\in \Cc'_E({\bf 1},{\bf 2})& u=\gamma_{21}(t')\in \Cc_E({\bf 1},{\bf 2})&
\psi=\alpha_{21}(v):\ M\ot H\to M\ot_BA\\
u'\in \Cc'_E({\bf 2},{\bf 1})& t=\gamma_{12}(u')\in \Cc_E({\bf 2},{\bf 1})&
\varphi=\alpha_{12}(v):\ M\ot_B A\to M\ot H\\
w'\in \Cc'_E({\bf 2},{\bf 2})& w=\gamma_{22}(w')\in \Cc_E({\bf 2},{\bf 2})&
\kappa=\alpha_{22}(w):\ M\ot_B A\to M\ot_BA
\end{array}$$
Furthermore, let $\Theta=\ol{\delta}_2(\theta)$ and $\phi=\ol{\delta}_1(\varphi)$,
see \leref{3.2}. The six remaining identities that we have to prove are
\begin{eqnarray}
\alpha_{21}(v*u)&=&\alpha_{21}(u)\circ\alpha_{11}(v)=\psi\circ\theta;\eqlabel{21a}\\
\alpha_{21}(w*u)&=&\alpha_{22}(w)\circ\alpha_{21}(u)=\kappa\circ\psi;\eqlabel{21b}\\
\alpha_{11}(t*u)&=&\alpha_{12}(t)\circ\alpha_{21}(u)=\varphi\circ\psi;\eqlabel{11}\\
\alpha_{12}(t*w)&=&\alpha_{12}(t)\circ\alpha_{22}(w)=\varphi\circ\kappa;\eqlabel{12a}\\
\alpha_{12}(v*t)&=&\alpha_{11}(v)\circ\alpha_{12}(t)=\theta\circ\varphi;\eqlabel{12b}\\
\alpha_{22}(u*t)&=&\alpha_{21}(u)\circ\alpha_{12}(t)=\psi\circ\varphi;\eqlabel{22}.
\end{eqnarray}

\equref{21a} is equivalent to $\ol{\beta}_{21}(\psi\circ\theta)=t'\star v'$. This can be shown
as follows
\begin{eqnarray*}
&&\hspace*{-2cm}
((t'\star v')(h))(m\ot_B a)=
(t'(h_{(2)})\circ v'(h_{(1)}))(m\ot_B a)\\
&=& (t'(h_{(2)}))\bigl( \Theta(m\ot h_{(1)})\ot_B a\bigr)\\
&=& \psi\bigl(\Theta(m\ot h_{(1)})\ot h_{(2)}\bigr)a\\
&=& (\psi\circ \theta)(m\ot h)a\\
&=& \bigl(\ol{\beta}_{21}(\psi\circ \theta)\bigr)(m\ot_B a).
\end{eqnarray*}

\equref{21b} is equivalent to $\beta_{21}(w'\star t')=\kappa\circ\psi$.\\
$\psi$ is given by the formula (see \leref{3.5}):
$$\psi(m\ot h)=t'(h)(m\ot_B 1).$$
$t'$ is right $H$-colinear, hence $\rho(t'(h))=t'(h_{(1)})\ot h_{(2)}$, and
$$\rho(\psi(m\ot h))=t'(h_{(1)})(m\ot_B 1)\ot h_{(2)}.$$
Then we have
\begin{eqnarray*}
&&\hspace*{-2cm}
(\kappa\circ\psi)(m\ot h)=
(w'(\psi(m\ot h)_{[1]}))(\psi(m\ot h)_{[0]})\\
&=&(w'(h_{(2)}))\bigl(t'(h_{(1)})(m\ot_B 1)\bigr)\\
&=& ((w'\star t')(h))(m\ot_B 1)=\beta_{21}(w'\star t')(m\ot h).
\end{eqnarray*}

\equref{11} is equivalent to $\ol{\beta}_{11}(\varphi\circ\psi)=u'\star t'$.\\
First observe that
\begin{eqnarray*}
&&\hspace*{-2cm}
(\ol{\beta}_{11}(\varphi\circ\psi)(h))(m\ot_B a)=
((M\ot \varepsilon)\circ\varphi\circ\psi)(m\ot h)\ot_B a\\
&=& (\phi\circ\psi)(m\ot h)\ot_B a.
\end{eqnarray*}
Now write
$$\psi(m\ot h)=\sum_j m_j\ot_N a_j.$$
Since $\psi$ is right $H$-colinear, we have
\begin{equation}\eqlabel{3.9.2}
\psi(m\ot h_{(1)})\ot h_{(2)}=\sum_j (m_j\ot_N a_{j[0]})\ot_N a_{j[1]}.
\end{equation}
Then we compute
\begin{eqnarray*}
&&\hspace*{-2cm}
\bigl((u'\star t')(h)\bigr)(m\ot_B a)=
\bigl(u'(h_{(2)})\circ t'(h_{(1)})\bigr)(m\ot_B a)\\
&=&
u'(h_{(2)})\bigl(\psi(m\ot h_{(1)})a\bigr)\\
&\equal{\equref{3.9.2}}&
\sum_j u'(a_{j[1]})\bigl(\psi(m_j\ot_N a_{j[0]})a\bigr)\\
&=&
\sum_{i,j} \phi\bigl(m_j \ot_B l_i(\ol{S}(a_{j[1]}))\bigr)\ot_B r_i(\ol{S}(a_{j[1]}))a_{j[0]}a\\
&\equal{\equref{1.2.6a}}&
\sum_j \phi(m_j\ot_B a_j)\ot_B a=(\phi\circ\psi)(m\ot h)\ot_B a.
\end{eqnarray*}

\equref{12a} is equivalent to $\beta_{12}(u'\star w')=\varphi\circ\kappa$.\\
We apply \leref{3.8} and write
$$\kappa(m\ot_B a)=w'(a_{[1]})(m\ot_B a_{[0]})=\sum_j m_j\ot_B a_j.$$
Since $\kappa$ is right $H$-colinear, we have
\begin{equation}\eqlabel{3.9.3}
\kappa(m\ot_B a_{[0]})\ot a_{[1]}=\sum_j (m_j\ot_Ba_{j[0]})\ot a_{j[1]}.
\end{equation}
Recall from \leref{3.6} that
$\phi(m\ot_B a)\ot_B 1=u'(a_{[1]})(m\ot_B a_{[0]})$. Then we compute
\begin{eqnarray*}
&&\hspace*{-2cm}
((M\ot \varepsilon)\circ \varphi \circ\kappa)(m\ot_B a)\ot_B 1=
(\phi\circ\kappa)(m\ot_B a)\ot_B 1\\
&=&\sum_j u'(a_{j[1]})(m\ot_B a_{j[0]})\\
&\equal{\equref{3.9.3}}&
u'(a_{[1]})(\kappa(m\ot_B a_{[0]}))\\
&=& \bigl(u'(a_{[2]})\circ w'(a_{[1]})\bigr)(m\ot_B a_{[0]})\\
&=&\bigl((u'\star w')(a_{[1]})\bigr)(m\ot_B a_{[0]})\\
&=& \bigl((M\ot \varepsilon)\circ \beta_{12}(u'\star w')\bigr)(m\ot_B a).
\end{eqnarray*}
It follows that
$\ol{\delta}_1(\varphi\circ\kappa)=(M\ot \varepsilon)\circ \varphi \circ\kappa=
(M\ot \varepsilon)\circ \beta_{12}(u'\star w')=\ol{\delta}_1(\beta_{12}(u'\star w'))$.
and then $\varphi\circ\kappa=\beta_{12}(u'\star w')$.\\

\equref{12b} is equivalent to $v*t=\ol{\alpha}_{12}(\theta\circ\varphi)$. Recall
from \equref{AA} that
$$(t(h))(m\ot_Ba)=\sum_i \phi(m\ot l_i(h))\ot_B r_i(h)a,$$
and from \leref{3.3} that
$$(v(h))(m\ot_Ba)=(v'(S(h)))(m\ot_Ba)=\Theta(m\ot S(h))\ot_B a.$$
Then we compute
\begin{eqnarray*}
&&\hspace*{-2cm}
((v* t)(h))(m\ot_B a)=
\bigl( v(h_{(1)}\circ v(h_{(2)})\bigr)(m\ot_B a)\\
&=& v(h_{(1)})\bigl( \sum_i \phi(m\ot l_i(h_{(2)}))\ot_B r_i(h_{(2)})a\bigr)\\
&=& \sum_i \Theta\bigl(\phi(m\ot l_i(h_{(2)}))\ot S(h_{(1)})\bigr) \ot_B r_i(h_{(2)})a\\
&\equal{\equref{1.2.4}}&
\sum_i \Theta\bigl(\phi(m\ot l_i(h)_{[0]})\ot l_i(h)_{[1]}\bigr) \ot_B r_i(h)a\\
&=& \sum_i \Theta(\varphi(m\ot l_i(h)) \ot_B r_i(h)a\\
&=& \sum_i ((M\ot \varepsilon)\circ \theta\circ \varphi)(m\ot l_i(h)) \ot_B r_i(h)a\\
&\equal{\equref{AA}}& (\alpha_{12}^{-1}(\theta\circ \varphi))(m\ot_B a).
\end{eqnarray*}

Finally, \equref{22} is equivalent to
$\beta_{22}(t'\star u')=\psi\circ\varphi$. From \leref{3.5}, we have that
$\psi(m\ot h)=t'(h)(m\ot_B 1)$, and from \leref{3.6} that
$\phi(m\ot_B a)\ot_B 1=u'(a_{[1]})(m\ot_B a_{[0]})$, hence
\begin{eqnarray*}
&&\hspace*{-2cm}
(\psi\circ\varphi)(m\ot_B a)=
\psi\bigl(\phi(m\ot_B a_{[0]})\ot a_{[1]}\bigr)\\
&=&
t'(a_{[1]})\bigl(\phi(m\ot_B a_{[0]})\ot_B 1\bigr)\\
&=&
(t'(a_{[2]})\circ u'(a_{[1]}))(m\ot_B a_{[0]})\\
&=&
((t'\star u')(a_{[1]}))(m\ot_B a_{[0]})\\
&=&
(\beta_{22}(t'\star u'))(m\ot_B a).
\end{eqnarray*}
\end{proof}

\section{The left-right case}\selabel{4}
Assume that $H$ is projective as a $k$-module.
Assume that $A$ is a left faithfully flat $H$-Galois extension of $B$, that is,
$A$ satisfies conditions (4) and (5) of \thref{1.1}. A left $A$-linear map between
left=right $(A,H)$-modules is called rational if there exists a (unique)
$f_{[0]}\ot f_{[1]}\in {}_A\Hom(P,Q)\ot H$ such that 
$\rho(f(p))=f_{[0]}(p_{[0]})\ot p_{[1]}f_{[1]}$. ${}_A\HOM(P,Q)$, the submodule of
rational maps is a right $H$-comodule and ${}_A\END(P)^{\rm op}$
is a right $H$-comodule algebra.\\
Now take $M\in {}_B\Mm$, and let $E={}_A\END(A\ot_BM)^{\rm op}$. Then
$F=E^{{\rm co}H}={}_A\End^H(A\ot_B M)^{\rm op}\cong {}_B\End(M)^{\rm op}$.
Let $\Ee_M$ be the full subcategory of ${}_B\Mm^H$ with objects $B\ot H$ and
$A\ot_B M$.

\begin{theorem}\thlabel{4.1}
With notation and assumptions as above, we have a duality
$\alpha:\ \Cc_E\to \Ee_M$.
\end{theorem}

\begin{proof}
Let $\alpha({\bf 1})=M\ot H$ and $\alpha({\bf 2})=A\ot_B M$. Below we present
the descriptions of the maps $\alpha_{ji}:\ \Cc_E({\bf i},{\bf j})\to \Dd_M({\bf j},{\bf i})$
and their inverses $\ol{\alpha}_{ji}$. All the other verifications are similar to corresponding
arguments in the proof of \thref{3.1} and are left to the reader.
Observe that we have
two natural isomorphisms
$$\delta_1:\  {}_B\Hom(A\ot_B M,M)\to {}_B\Hom^H(A\ot_B M, M\ot H);$$
$$\delta_2:\ {}_B\Hom(M\ot H,M)\to {}_B\End^H(M\ot H)$$
defined as follows:
$$\delta_1(\phi)(a\ot_B m)=\phi(a_{[0]}\ot_B m)\ot a_{[1]}~~;~~
\ol{\delta}_1(\varphi)=(M\ot\varepsilon)\circ\varphi;$$
$$\delta_2(\Theta)(m\ot h)=
\Theta(m\ot h_{(1)})\ot h_{(2)}~~;~~
\ol{\delta}_2(\theta)=(M\ot\varepsilon)\circ\theta.$$
We have an isomorphism
$$\tilde{\alpha}_{ 11}:\ \Cc_E({\bf 1},{\bf 1})=\Hom(H,E^{{\rm co}H})\to {}_B\Hom(M\ot H,M),$$
given by the formulas
\begin{eqnarray*}
&&1\ot_B \tilde{\alpha}_{ 11}(v)(m\ot h)=v(h)(1\ot_B m);\\
&&\hat{\alpha}_{ 11}(\Theta)(h)(a\ot_B m)=a\ot_B \Theta(m\ot h).
\end{eqnarray*}
We then define $\alpha_{ 11}=\beta_2\circ \tilde{\alpha}_{\bf 11}$.\\
The isomorphism
 \[\alpha_{ 12}:\ \Cc_E({\bf 2},{\bf 1})=\Hom^H(H,E)\to {}_B\Hom^H(M\ot
H,A\ot_B M)\]
is given by the formulas
$$\alpha_{ 12}(t)(m\ot h)=t(h)(1\ot_Bm)~~;~~(\ol{\alpha}_{ 12}(\psi)(h))(a\ot_B m)=
a\psi(m\ot h).$$
We have an isomorphism
$$\tilde{\alpha}_{ 21}:\ \Cc_E({\bf 1},{\bf 2})\to {}_B\Hom(A\ot_BM,M),$$
given by the formulas
\begin{eqnarray*}
&&1\ot_B \tilde{\alpha}_{ 21}(u)(a\ot_B m)=u(a_{[1]})(a_{[0]}\ot_Bm);\\
&& \bigl(\hat{\alpha}_{ 21}(\phi)(h)\bigr)(a\ot_B m)= \sum_i
al_i(h)\ot_B \phi(r_i(h)\ot_B m).
\end{eqnarray*}
We then define $\alpha_{ 21}=\beta_1\circ \tilde{\alpha}_{\bf 21}$.\\
Finally, the isomorphism
$$\alpha_{ 22}:\ \Cc_E({\bf 2},{\bf 2})\to {}_B\End^H(A\ot_BM)$$
is given by the formulas
\begin{eqnarray*}
&&\alpha_{ 22}(w)(a\ot_B m)=w(a_{[1]})(a_{[0]}\ot_B m);\\
&&(\ol{\alpha}_{ 22}(\kappa))(h)(a\ot_B m)=
\sum_i al_i(h)\kappa(r_i(h)\ot_B m).
\end{eqnarray*}
\end{proof}

\section{Cleft extensions}\selabel{5}
Recall that a right $H$-comodule algebra $A$ is called cleft if there exists a convolution
invertible $t\in \Hom^H(H,A)$. This means precisely that ${\bf 1}$ and ${\bf 2}$ are
isomorphic objects in $\Cc_A$.\\
There is a Structure Theorem for cleft extensions, see \cite{DT} or \cite[Theorem 7.2.2]{Mont}:
cleft extensions are precisely the crossed product.
We will present a proof of this Theorem, based on the duality from \thref{4.1}. First let us recall
the precise definition of a crossed product, following \cite[Sec. 7.1]{Mont}.\\

Let $H$ be a Hopf algebra measuring an algebra $B$: this means that we have a map
$\omega:\ H\ot B\to B$, $\omega(h\ot b)=h\cdot b$ such that $h\cdot 1=\varepsilon(h)1$
and $h\cdot (bc)=(h_{(1)}\cdot b)(h_{(2)}\cdot c)$, for all $h\in H$ and $b,c\in B$. Let
$\sigma:\ H\ot H\to B$ be a map with convolution inverse $\ol{\sigma}$. $A\#_\sigma H$
is $A\# H$ with multiplication
\begin{equation}\eqlabel{5.1.1}
(b\# h)(c\# k)=b(h_{(1)}\cdot c)\sigma (h_{(2)}\ot k_{(1)})\# h_{(3)}k_{(2)}.
\end{equation}

The following result originates from \cite{BCM,DT}, see also \cite[Lemma 7.1.2]{Mont}.
The proof is straightforward.

\begin{proposition}\prlabel{5.1}
With notation as above, $B\#_\sigma H$ is an associative algebra with unit $1\# 1$
if and only if the following conditions hold:\\
1) $B$ is a twisted $H$-module, this means that $1\cdot b=b$, for all $b\in B$, and
\begin{equation}\eqlabel{5.1.2}
h\cdot(k\cdot b)=\sigma(h_{(1)}\ot k_{(1)})((h_{(2)}k_{(2)})\cdot b)\ol\sigma(h_{(3)}\ot k_{(3)}),
\end{equation}
for all $h,k\in H$ and $b\in B$;\\
2) $\sigma$ is a normalized cocycle; this means that $\sigma(h\ot 1)=\sigma(1\ot h)=
\varepsilon(h)1$ and
\begin{equation}\eqlabel{5.1.3}
\bigl(h_{(1)}\cdot \sigma(k_{(1)}\ot l_{(1)})\bigr)\sigma(h_{(2)}\ot k_{(2)}l_{(2)})
=\sigma(h_{(1)}\ot k_{(1)})\sigma(h_{(2)}k_{(2)}\ot l),
\end{equation}
for all $h,k,l\in H$. Then $B\#_\sigma H$ is called a crossed product; it is an $H$-comodule
algebra, with coaction induced by the comultiplication on $H$.
\end{proposition}

Now we present the Structure Theorem for cleft $H$-comodule algebras. But first we make
the following remark. Assume that $t\in \Hom^H(H,A)$ has convolution inverse $u$.
Then $t(1)u(1)=u(1)t(1)=1$. Then $t'=u(1)t\in \Hom^(H,A)$ has convolution inverse $u t(1)$,
and satisfies $t'(1)=1$. So if $A$ is cleft, then there exists a convolution invertible $t\in \Hom^H(H,A)$
taking the value $1$ in $1$.

\begin{theorem}\thlabel{5.2}
Let $H$ be a projective Hopf algebra, $A$ a right $H$-comodule algebra, and
$B=A^{{\rm co}H}$. Then the following assertions are equivalent:
\begin{enumerate}
\item $A$ is cleft;
\item $A$ is isomorphic to a crossed product $B\#_\sigma H$;
\item $A$ is a faithfully flat left Hopf-Galois extension of $B$, and $A$ is isomorphic to
$B\ot H$ as a left $B$-module and a right $H$-comodule.
\end{enumerate}
\end{theorem}

\begin{proof}
$(1)\Longrightarrow (2)$. \thref{4.1} holds under the assumption that $A$ is an
$H$-Galois extension. However, if $M\in {}_B\Mm$ is such that $\eta'_M$ is an isomorphism,
then we we still have the functor $\alpha$. This happens in the particular situation where
$M=B$. In this case $E={}_A\END(A\ot_BB)^{\rm op}={}_A\END(A)^{\rm op}\cong A$,
and $F=E^{{\rm co}H}=A^{{\rm co}H}=B$.\\
If $A$ is cleft, then there exists a convolution invertible $t\in \Hom^H(H,E)$, with $t(1)=1$,
and then
$\alpha_{12}(t):\ B\ot H\to A\ot_B B=A$ is an isomorphism in ${}_B\Mm^H$. We transport the multiplication
on $A$ to $B\ot H$, and write $B\#_\sigma H$ for $A\ot H$ with this multiplication. We can easily
make this explicit: with notation as in \thref{4.1}, let $\alpha_{12}(t)=\psi$, $u$ the convolution inverse
of $t$, $\tilde{\alpha}_{21}(u)=\phi$ and $\alpha_{21}(u)=\varphi$. Using the formulas in the proof
of \thref{4.1}, we find
$$\psi(b\ot h)=bt(h)~~;~~\phi(a)=a_{[0]}u(a_{[1]})~~:~~\varphi(a)=a_{[0]}u(a_{[1]})\ot a_{[2]}.$$
Now we transport the multiplication:
\begin{eqnarray*}
&&\hspace*{-2cm}
(b\# h)(c\# k)=\varphi\bigl(\psi(b\# k)\psi(c\# k)\bigr)
=\varphi(bt(h)ct(k))\\
&=&
bt(h_{(1)})ct(k_{(1)})u(h_{(2)}k_{(2)})\ot h_{(3)}k_{(3)}\\
&=&
bt(h_{(1)})cu(h_{(2)})t(h_{(3)})
t(k_{(1)})u(h_{(4)}k_{(2)})\ot h_{(5)}k_{(3)}
\end{eqnarray*}
Now define
\begin{equation}\eqlabel{omega}
\omega_t:\ H\ot B\to B,~~\omega_t(h\ot b)=t(h_{(1)})bu(h_{(2)})=h\cdot b,
\end{equation}
and
$$\sigma:\ H\ot H\to B,~~\sigma(h\ot k)=t(h_{(1)})t(k_{(1)})u(h_{(2)}k_{(2)}).$$
Then the multiplication is given by formula \equref{5.1.1}. The unit of the multiplication
is $\varphi(1)=u(1)\#1=1\#1$. It is obvious that $\omega_t$ measures $B$ and that
$\sigma$ is convolution invertible, with inverse
$\ol{\sigma}(h\ot k)=t(h_{(1)}k_{(1)})u(k_{(2)})u(h_{(2)})$. Straightforward computations
show that the conditions of \prref{5.1} are satisfied, so $A$ is isomorphic to the crossed
product $B\#_\sigma H$.\\
$(2)\Longrightarrow (3)$. Consider a crossed product $A=B\#_\sigma H$, as in \prref{5.1}.
Since $H$ is projective, and therefore faithfully flat, as a $k$-module, $A$ is faithfully 
flat as a left and right $B$-module.
Now $A\ot_B A=(B\ot H)\ot_B (B\ot H)\cong B\ot H\ot H$, and then it is easy to see that the
canonical map $\can:\ B\ot H\ot H\to B\ot H\ot H$ is given by the formula
$$\can(a\ot b\ot k)=a\sigma(h_{(1)}\ot k_{(1)})\ot h_{(2)}k_{(2)}\ot k_{(3)}.$$
$\can$ is bijective, with inverse
$$\can^{-1}(a\ot b\ot k)= a\ol{\sigma}(h_{(1)}S(k_{(2)})\ot k_{(3)})\ot h_{(2)}S(k_{(1)})\ot k_{(4)}.$$
Then $\can'$ is also bijective, and $A$ is a faithfully flat left and right $H$-Galois extension,
clearly isomorphic to $B\ot H$ as a left $B$-module and a right $H$-comodule.\\
$(2)\Longrightarrow (3)$. Since $A$ is a faithfully flat left $H$-Galois extension, we can apply
 \thref{4.1}. We have an isomorphism $\psi:\ B\ot H\to A$ in ${}_B\Mm^H$, and $t=\alpha_{12}(\psi)$
 is then a convolution invertible element in $\Hom^H(H,A)$. This shows that $A$ is cleft.
\end{proof}

\begin{remark}\relabel{5.3}
Let $A=B\#_\sigma H$ be a crossed product. From the formulas in \thref{4.1}, we can explicitly
compute $t=\alpha_{12}(\psi)$ and $u=\hat{\alpha}_{12}(\phi)$. First, $\psi:\ B\ot H\to A=B\#_\sigma H$
is the identity map, and then we see easily that $t(h)=1\# h$. In the proof of $(2)\Longrightarrow (3)$,
we constructed the inverse of the canonical map, and from this we deduce that
$$\sum_i l_i(h)\ot r_i(h)=
\Bigl(\ol{\sigma}\bigl(S(h_{(2)})\ot h_{(3)}\bigr)1_B\# S(h_{(1)})\Bigr)\ot_B \Bigl(1_B\# h_{(4)}\Bigr).$$
Now we have that $\phi=(B\ot \varepsilon):\ A=B\#_\sigma H\to B$, and then we see that
$$u(h)=\ol{\sigma}\bigl(S(h_{(2)})\ot h_{(3)}\bigr)1_B\# S(h_{(1)}).$$
Of course these formulas are well-known, see for example \cite[Prop. 7.2.7]{Mont}.
\end{remark}

If $t\in \Hom^H(H,A)$ is an algebra map, then $t$ is convolution invertible, with convolution inverse
$t\circ S$. Then the cocycle $\sigma$ constructed in the proof of \thref{5.2} is trivial,
and \equref{5.1.2} reduces to
$h\cdot(k\cdot b)=(hk)$, so that $B$ is an $H$-module algebra. Then $A$ is isomorphic to the
smash product $B\# H$. This proves $(1)\Longrightarrow (2)$ in the next theorem.

\begin{theorem}\thlabel{5.4}
Let $H$ be a projective Hopf algebra, $A$ a right $H$-comodule algebra, and
$B=A^{{\rm co}H}$. Then the following assertions are equivalent:
\begin{enumerate}
\item there exists an algebra map $t\in \Hom^H(H,A)$;
\item $A$ is isomorphic to a smash product $B\# H$.
\end{enumerate}
\end{theorem}

\begin{proof}
$(2)\Longrightarrow (1)$. The map $t$ constructed in \reref{5.3} is an algebra map.
\end{proof}

 Consider the space
\[\Omega_A=\{t\in \Hom^H(H,A)~|~t~{\rm
is~an~algebra~map}\}.\]
We have the following equivalence
relation on $\Omega_A$: $t_1\sim t_2$ if and only if there exists $b\in U(B)$
such that $bt_1(h)=t_2(h)b$, for all $h\in H$. We denote
$\ol{\Omega}_A=\Omega_A/\sim$. With some extra assumptions, we can give
a categorical and cohomological interpretation of $\Omega_A$ and $\ol{\Omega}_A$.
Throughout the rest of this Section, we will assume that $H$ is cocommutative, $B$
is commutative and $A$ is cleft. In this situation $\Cc_A({\bf 2},{\bf 2})=
\Hom(H,B)$.
For a convolution invertible $t\in \Hom^H(H,A)$,
we consider the map $\omega_t$, see \equref{omega}.

\begin{lemma}\lelabel{5.5}
$\omega_t$ is independent of the choice of $t$, and makes $B$ into a left
$H$-module algebra.
\end{lemma}

\begin{proof}
The second statement follows immediately from \equref{5.1.3}, taking into account that
$B$ is commutative. Let $t,t_0\in \Hom^H(H,A)$ be convolution invertilble,
with convolution inverses $u$ and $u_0$. Using the
commutativity of $B$ again, we find
$$u_0(h_{(1)})t(h_{(2)})bu(h_{(3)})t_0(h_{(3)})=
bu_0(h_{(1)})t(h_{(2)})u(h_{(3)})t_0(h_{(3)})=b.$$
Then
\begin{eqnarray*}
&&\hspace*{-2cm}
w_{t_0}(h\ot b)=t_0(h_{(1)})bu_0(h_{(2)})\\
&=& t_0(h_{(1)})u_0(h_{(2)})t(h_{(3)})bu(h_{(4)})t_0(h_{(5)})u_0(h_{(6)})\\
&=& t(h_{(1)})bu(h_{(2)})=w_{t}(h\ot b).
\end{eqnarray*}
\end{proof}

Since $B$ is a left $H$-module algebra, we can consider the Sweedler cohomology
groups $H^n(H,B)$ with values in $B$, see \cite{Sweedler}.

\begin{theorem}\thlabel{5.6}
Assume that $H$ is cocommutative, $B$ is commutative and $H$ is cleft. Then we have
the following subcategory $\Xx_A$ of $\Cc_A$. $\Xx_A$ has two objects ${\bf 1}$
and ${\bf 2}$, and
\begin{eqnarray*}
\Xx_A({\bf 1},{\bf 1})&=& Z^1(H,B);\\
\Xx_A({\bf 2},{\bf 1})&=& \Omega_A;\\
\Xx_A({\bf 2},{\bf 2})&=&\{\omega\in \Hom(H,B)~|~\omega\circ S\in Z^1(H,B)\};\\
\Xx_A({\bf 1},{\bf 2})&=&\{t\circ S~|~t\in \Omega\}.
\end{eqnarray*}
\end{theorem}

\begin{proof}
Recall that a convolution invertible $v:\ H\to B$ is a 1-cocycle in $Z^1(H,B)$ if
$$v(hk)=(h_{(1)}\cdot v(k))v(h_{(2)}),$$
for all $h,k\in H$. A convolution invertible $w:\ H\to B$ lies in $\Xx_A({\bf 2},{\bf 2})$
if
$$w(hk)=(S(k_{(1)})\cdot w(h))w(h_{(2)}),$$
for all $h,k\in H$. It is well-known that $\Xx_A({\bf 1},{\bf 1})= Z^1(H,B)$
and $\Xx_A({\bf 2},{\bf 2})$ are groups. Take $v\in Z^1(H,A)$, $w=v\circ S\in \Xx_A({\bf 2},{\bf 2})$,
$t,t'\in \Omega_A$, $u=t\circ S, u'=t'\circ S\in \Xx_A({\bf 1},{\bf 2})$.\\
1) $t*u_1\in Z^1(H,B)$: for all $h,k\in H$, we have
\begin{eqnarray*}
&&\hspace*{-2cm}
(t_1*u)(hk)=t(h_{(1)})t(k_{(1)})u_1(k_{(2)})u_1(h_{(2)})\\
&=& t(h_{(1)}) (t*u_1)(k)u(h_{(2)})t(h_{(3)})u_1(h_{(4)})\\
&=& (h_{(1)}\cdot (t*u_1)(k))(t*u)(k).
\end{eqnarray*}
2) $v*t\in \Omega_A$: for all $h,k\in H$, we have
\begin{eqnarray*}
&&\hspace*{-2cm}
(v*t)(hk)=(h_{(1)}\cdot v(k_{(1)}))v(h_{(2)})t(h_{(3)}t(k_{(2)})\\
&=& t(h_{(1)})v(k_{(1)})u(h_{(2)}) v(h_{(3)})t(h_{(4)}t(k_{(2)})\\
&=& t(h_{(1)})u(h_{(2)}) v(h_{(3)})t(h_{(4)}v(k_{(1)})t(k_{(2)})\\
&&\hspace*{1cm}(B~{\rm is~commutative})\\
&=& (v*t)(h)(v*t)(k).
\end{eqnarray*}
3) $t*w\in \Omega_A$: for all $h,k\in H$, we have
\begin{eqnarray*}
&&\hspace*{-2cm}
(t*w)(hk)= t(h_{(1)})t(k_{(1)}) (S(k_{(2)})\cdot w(h_{(2)}))w(k_{(3)})\\
&=& t(h_{(1)})t(k_{(1)}) u(k_{(2)}) w(h_{(2)})t(k_{(3)})w(k_{(3)})\\
&=& (t*w)(h)(t*w)(k).
\end{eqnarray*}
4) We know from 1) that $t*u_1\in Z^1(H,B)$, hence $(t*u_1)\circ S= u*t_1\in \Xx_A({\bf 2},{\bf 2})$.\\
5) We know from 2) that $v*t\in \Omega_A$, hence $(v*t)\circ S=w*u\in \Xx_A({\bf 1},{\bf 2})$.\\
6) We know from 3) that $t*w\in \Omega_A$, hence $(t*w)\circ S=u*v\in \Xx_A({\bf 1},{\bf 2})$.
\end{proof}

Obviously $\Xx_A$ is a groupoid: every morphism in $\Xx_A$ is invertible. Assume now that
$\Omega_A\neq \emptyset$, and fix $t_0\in \Omega_A$. Then the map $F:\ Z^1(H,B)\to \Omega_A$,
$F(v)=v*t_0$ is a bijection. The inverse is given by $F^{-1}(t)=t*u_0$, with $u_0=t_0\circ S$.

\begin{proposition}\prlabel{5.7}
$F$ sends equivalence classes in $Z^1(H,B)$ to equivalence classes in $\Omega_A$, and a similar
property holds for $F^{-1}$. Hence $F$ induces a bijection $H^1(H,B)\to \ol{\Omega}_A$.
\end{proposition}

\begin{proof}
For each invertible $b\in B$, we have a $1$-cocycle $f_b:\ H\to B$, $f_b(h)=(h\cdot b)b^{-1}$.
Then $B^1(H,B)=\{f_b~|~b\in U(B)$, and $H^1(H,B)=Z^1(H,B)/B^1(H,B)$. First assume that
$v\sim v_1$ in $Z^1(H,B)$. Then there exist $b\in U(B)$ such that
$v=f_b* v_1$. Let $F(v)=t$, $F(v_1)=t_1$, then
$$t=v*t_0=f_b*v_1*t_0=f_b*t_1$$
and
$$
t(h)= b^{-1}(h_{(1)}\cdot b)t_1(h_{(2)})=b^{-1}t_1(h_{(1)})bu_1(h_{(2)})t_{(1)}(h_{(3)})
=b^{-1}t_1(h)b,$$
for all $h\in H$, so that $t\sim t_1$. Conversely, if $t\sim t_1$, then there exists
$b\in U(B)$ such that $t(h)=b^{-1}t_1(h)b$, for all $h\in H$, and
\begin{eqnarray*}
&&\hspace*{-2cm}
(t* u_0)(h)= b^{-1} t_1(h_{(1)}) b u_0(h_{(2)})
= b^{-1} t_1(h_{(1)}) b u(h_{(2)})  t(h_{(3)})u_0(h_{(4)})\\
&=& b^{-1} (h_{(1)}\cdot b) (t_1* u_0)(h_{(2)})= (f_b* t_1*u_0)(h),
\end{eqnarray*}
for all $h\in H$, and then $t* u_0$ is cohomologous to $t_1* u_0$.
\end{proof}

\section{Stable modules and the Militaru-\c Stefan lifting Theorem}\selabel{6}
We return to the setting of \seref{3}: $A$ is a right faithfully flat $H$-Galois extension,
$B$ is the subalgebra of coinvariants, and $M$ is a right $B$-module.
Recall from \cite{Schneider1} that $M$ is called $H$-stable if $M\ot H$ and
$M\ot_B A$ are isomorphic as right $B$-modules and right $H$-comodules.
From \thref{3.1}, we immediately obtain the following result, originally due to
Schneider \cite{Schneider1} in the case where $H$ is finitely generated and projective,
and to
Militaru and \c Stefan, \cite[Lemma 3.2]{MiSt} in the general case.

\begin{proposition}\prlabel{6.1}
$M\in \Mm_B$ is $H$-stable if and only if $E=\END_A(M\ot B A)$ is cleft, that is,
there exists an $H$-colinear convolution invertible $t:\ H\to E$.
\end{proposition}

As we have seen in \seref{5}, an $H$-colinear algebra map is convolution invertible.
Militaru and \c Stefan proved that the existence of an $H$-colinear algebra map
$t:\ H\to E$ is equivalent to the existence of an associative action of $A$ and $M$
extending the right $B$-action. This can also be derived from \thref{3.1}, which is
what we will now discuss.
We fix the following notation: $\phi:\ M\ot_B A\to A$ is a right $B$-linear map,
$\varphi=\delta_1(\phi)$, $\hat{\beta}_{12}(\phi)=u'$, $t=u\circ S=\hat{\alpha}_{12}(\phi)$.
We also write $\phi(m\ot_B a)=m\cdot a$. From \leref{3.6}, we recall the following formulas
(see (\ref{eq:3.6.0}-\ref{eq:3.6.1}):
\begin{eqnarray}
m\cdot a\ot_B 1&=& u'(a_{[1]})(m\ot_B a_{[0]});\eqlabel{6.2.1}\\
t(h)(m\ot_B a)&=& \sum_i \phi(m\ot_B l_i(h))\ot_B r_i(h)=\sum_i m\cdot l_i(h)\ot_B r_i(h)
\eqlabel{6.2.2}.
\end{eqnarray}
We then immediately have the following result:

\begin{proposition}\prlabel{6.2}
With notation as above, the following assertions are equivalent:
\begin{enumerate}
\item $t(1)=1$;
\item $u'(1)=1$;
\item $m\cdot 1=1$.
\end{enumerate}
\end{proposition}

\begin{proof} $(1)\Longrightarrow (2)$ is obvious. $(2)\Longrightarrow (3)$ follows immediately
from \equref{6.2.1}, and $(3)\Longrightarrow (1)$ follows from \equref{6.2.2}.
\end{proof}

\begin{proposition}\prlabel{6.3}
With notation as above, the following assertions are equivalent:
\begin{enumerate}
\item $t$ is multiplicative;
\item $u$ is anti-multiplicative;
\item the right $A$-action on $M$ defined by $\phi$ is associative.
\end{enumerate}
\end{proposition}

\begin{proof}
$(1)\Longrightarrow (2)$ is obvious.\\
$(2)\Longrightarrow (3)$. For all $m\in M$ and $a,b\in A$, we have
\begin{eqnarray*}
&&\hspace*{-2cm}
(m\cdot (ab))\ot_B 1\equal{\equref{6.2.1}} u'(a_{[1]}b_{[1]})(m\ot_B a_{[0]}b_{[0]})\\
&=& \bigl((u'(b_{[1]})\circ u'(a_{[1]})\bigr)(m\ot_B a_{[0]}b_{[0]})\\
&=& u'(b_{[1]}) \bigl( u'(a_{[1]})(m\ot_B a_{[0]})b_{[0]}\bigr)\\
&\equal{\equref{6.2.1}}& 
u'(b_{[1]}) (m\cdot a\ot_B b_{[0]}
\equal{\equref{6.2.1}} (m\cdot a)\cdot b.
\end{eqnarray*}
$(3)\Longrightarrow (1)$. For all $h,k\in H$, $m\in M$ and $a\in A$, we have
\begin{eqnarray*}
&&\hspace*{-2cm}
t(hk)(m\ot_B a)\equal{\equref{6.2.2}}
\sum_i m\cdot l_i(hk)\ot_B r_i(hk)\\
&\equal{\equref{1.2.7}}&
\sum_{i,j} m\cdot (l_i(k)l_j(h))\ot_B r_j(h)r_i(k)\\
&=&
\sum_{i,j} (m\cdot l_i(k))\cdot l_j(h)\ot_B r_j(h)r_i(k)\\
&\equal{\equref{6.2.2}}&
\sum_i t(h)(m\cdot l_i(k)\ot_B r_i(k))\\
&\equal{\equref{6.2.2}}& (t(h)\circ t(k))(m\ot_B a).
\end{eqnarray*}
\end{proof}

Combining these results, we obtain the Militaru-\c Stefan lifting Theorem, see \cite[Theorem 2.3]{MiSt}.

\begin{theorem}\thlabel{6.4}
With notation as above, the following assertions are equivalent:
\begin{enumerate}
\item $t$ is an algebra map;
\item $u$ is an anti-algebra map;
\item $\phi$ makes $M$ into a right $B$-module.
\end{enumerate}
\end{theorem}

Now consider the set $\Lambda_M$ consisting of all right $B$-linear maps $\phi:\ M\ot_B A\to M$
defining a right $A$-module structure on $M$. It follows from \thref{6.4} that $\hat{\alpha}_{12}:\
\Lambda_M\to \Omega_E$ is a bijection. $\phi_1,\phi_2\in \Lambda_M$ are called equivalent
if the resulting right $A$-modules $M_1$ and $M_2$ are isomorphic. Let $\ol{\Lambda}$
be the quotient set.

\begin{proposition}\prlabel{6.5} \cite[Theorem 2.6]{MiSt}
Let $\phi_1,\phi_2\in \Lambda_M$, and $t_1=\hat{\alpha}_{12}(\phi_1)$,
$t_2=\hat{\alpha}_{12}(\phi_2)$ the corresponding $H$-colinear algebra maps
$H\to E$. Then $\phi_1\sim \phi_2$ if and only if $t_1\sim t_2$. Consequently
$\ol{\Omega}_E\cong \ol{\Lambda}$ classifies the isomorphism classes of right
$A$-module structures on $M$ extending the right $B$-action on $M$.
\end{proposition}

\begin{proof}
Let $M_i=M$ with right $A$-action $m\cdot_i a=\phi_i(m\ot_B a)$, and $u'_i=t_i\circ S^{-1}$
Recall from \seref{5}
that $t_1\sim t_2$ if and only if there exists an invertbile $f\in \End_B(M)\cong E^{{\rm co} H}$
such that
\begin{equation}\eqlabel{6.5.1}
t_1(h)\circ (f\ot_B A)=(f\ot_B A)\circ t_2(h),
\end{equation}
or, equivalently,
\begin{equation}\eqlabel{6.5.2}
u'_1(h)\circ (f\ot_B A)=(f\ot_B A)\circ u'_2(h),
\end{equation}
$\phi_1\sim \phi_2$ if and only if there exists an invertible $f\in \End_B(M)$ such that
$f(m\cdot_2 a)=f(m)\cdot_1 a$, for all $m\in M$ and $a\in A$.\\
If $t_1\sim t_2$ then
\begin{eqnarray*}
&&\hspace*{-2cm}
f(m\cdot_2 a)\ot_B 1\equal{\equref{3.6.0}}
((f\ot_B A)\circ u'_2(a_{[1]}))(m\ot_B a_{[0]}\\
&\equal{\equref{6.5.1}}&
(u'_1(a_{[1]})\circ (f\ot_B A))(m\ot_B a_{[0]}
= f(m)\cdot_1 a\ot_B 1,
\end{eqnarray*}
and it follows that $\phi_1\sim \phi_2$. Conversely, if $\phi_1\sim \phi_2$, then
\begin{eqnarray*}
&&\hspace*{-2cm}
((f\ot_B A)\circ t_2(h))(m\ot_B a)\equal{\equref{3.6.1}}
\sum_i f(m\cdot_2 l_i(h))\ot_B r_i(h)\\
&\equal{\equref{6.5.2}}&
\sum_i f(m)\cdot_1 l_i(h)\ot_B r_i(h)
= (t_1(h)\circ (f\ot_B A))(m\ot_Ba),
\end{eqnarray*}
and it follows that $t_1\sim t_2$.
\end{proof}

If $H$ is cocommutative, $\End_B(M)$ is commutative and $\Omega_E\neq\emptyset$, then
we can apply \prref{5.7}, and we obtain a cohomological description of $\ol{\Omega}_E$,
namely $\ol{\Omega}_E\cong \ol{\Lambda}_M\cong H^1(H,\End_B(M))$. This result is one
of the key arguments in \cite{CM}.

\end{document}